\documentclass[12pt,reqno]{amsart}
\usepackage{geometry}                
\usepackage{amsmath,amssymb}
\usepackage{graphicx}
\usepackage{amssymb,latexsym, amsmath}
\usepackage{amsfonts,amsthm}
\usepackage{amscd}
\usepackage{mathrsfs}
\usepackage{hyperref}
\usepackage{eucal}
\usepackage{epstopdf}
\usepackage{amsgen}
\usepackage{xspace}
\usepackage{verbatim}
\usepackage{stmaryrd}
\usepackage{enumitem}
\newlist{steps}{enumerate}{1}
\setlist[steps, 1]{label = Step \arabic*:}

\newcommand{\gd}{\Delta}

\newcommand{\inpt}[1]{\langle #1 \rangle}

\newcommand{\gw}{\Omega}
\newcommand{\ap}{\alpha}

\newcommand{\gb}{\beta}
\newcommand{\G}{\Gamma}

\newcommand{\ms}{\mathscr}

\newcommand{\nb}{\nabla}

\newcommand{\pdr}{\partial}

\newcommand{\beq}{\begin{equation}}
\newcommand{\eeq}{\end{equation}}
\newcommand{\bea}{\begin{align}}
\newcommand{\eea}{\end{align}}
\newcommand{\bthm}{\begin{theorem}}
\newcommand{\ethm}{\end{theorem}}
\newcommand{\bpr}{\begin{proof}}
\newcommand{\epr}{\end{proof}}
\newcommand{\bcl}{\begin{corollary}}
\newcommand{\ecl}{\end{corollary}}
\newcommand{\bpn}{\begin{proposition}}
\newcommand{\epn}{\end{proposition}}
\newcommand{\bre}{\begin{remark}}
\newcommand{\ere}{\end{remark}}
\newcommand{\bdf}{\begin{definition}}
\newcommand{\edf}{\end{definition}}
\newcommand{\bss}{\begin{align*}}
\newcommand{\ess}{\end{align*}}

\newcommand{\bl}{\label}

\newcommand{\mR}{\mathbb{R}}
\newcommand{\mH}{\mathbb{H}}

\newcommand{\mE}{\mathbb{E}}

\newtheorem{theorem}{Theorem}[section]
\newtheorem{corollary}[theorem]{Corollary}

\newtheorem{proposition}[theorem]{Proposition}

\theoremstyle{definition}
\newtheorem{definition}[theorem]{Definition}
\theoremstyle{remark}
\newtheorem{remark}{Remark}

\numberwithin{equation}{section}

\begin{document}

\title[Synchronization of Complex Neural Networks]{Dynamics and Synchronization of Complex Neural Networks with Boundary Coupling}

\author[C. Phan]{Chi Phan}
\address{Department of Mathematics and Statistics, University of South Florida, Tampa, FL 33620, USA}
\email{chi@mail.usf.edu}
\thanks{}

\author[L. Skrzypek]{Leslaw Skrzypek}
\address{Department of Mathematics and Statistics, University of South Florida, Tampa, FL 33620, USA}
\email{skrzypek@usf.edu}
\thanks{}

\author[Y. You]{Yuncheng You}
\address{Department of Mathematics and Statistics, University of South Florida, Tampa, FL 33620, USA}
\email{you@mail.usf.edu}
\thanks{}


\date{April 16, 2020}



\begin{abstract} 
A new mathematical model for complex neural networks of the partly diffusive Hindmasrh-Rose equations with boundary coupling is proposed. Through analysis of absorbing dynamics for the solution semiflow, the asymptotic synchronization of the complex neuronal networks at a uniform exponential rate is proved under the condition that stimulation signal strength of the ensemble boundary coupling exceeds a quantitative threshold expressed by the biological parameters. 
\end{abstract}

\maketitle

\section{\textbf{Introduction}}

Synchronization of biological neuron firing and bursting plays an important role in processing information and executing commands by the complex neural networks in the brain and central nerve system. Understanding of the synchronization and desynchronization mechanisms in biological neural networks by means of mathematical models and analysis is one of the central topics for advancing the researches in neuroscience and medical science, even in the theory of artificial neural networks and artificial intelligence. Increased and fast synchronization may lead to enhanced functionality and performance of central neuronal system or may lead to functional disorders of neuron systems such as epilepsy and Parkinson's disease \cite{HBB, JC, L, TBSS}.

Neurons are the nerve cells which form the major pathways of communication and ensemble networks capable to process, coordinate, and integrate biochemical and bio-electric synaptic signals. Approximately 100 billion neurons can be found in human nervous system and they are connected with approximately $10^{14}$ synapses. Neuron signals are short electric pulses called spikes or action potential. The synaptic pulse inputs received by neuron dendrites modify the intercellular membrane potentials and may cause bursting in alternating phases of rapid firing and then refractory quiescence. Neuron signals triggered at the axon hillock can propagate along the axon and transmit to the neighbor neurons. In a complex neural network, stimulation signals through synaptic couplings in ensemble neuron activities have to reach certain threshold for achieving synchronization \cite{DJ, SK}.

There are several mathematical models in ODEs to describe single neuron dynamics. The four-dimensional Hodgkin-Huxley model \cite{HH} (1952) is highly nonlinear and consists of the membrane potential equation combined with three ionic current equations of sodium, potassium, and others called leakage. The simplified two-dimensional FitzHugh-Nagumo model \cite{FH, NM} (1961-1962) with the membrane potential variable and the ion current variable features an effective phase-plane analysis to characterize the refractorily periodic excitation of neurons in response to suprathreshold input pulse, but this 2D model does not have chaotic solutions so that no self-sustained chaotic bursting can be shown. 

Another type of models is the three-dimensional Hindmarsh-Rose equations (1984) initiated in \cite{HR} and the diffusive as well as stochastic Hindmarsh-Rose equations recently proposed and studied on topics of regular and chaotic bursting dynamics, global attractors, and random attractors, cf. \cite{DH, ET, Hetal, IG, Phan, PY, PYY, SPH, WS} and the references therein.

Investigation of synchronization for biological neural networks has been conducted by using several mathematical models and methods. Most published results dealt with the FitzHuigh-Nagumo neuron networks with coupling by the gap junctions or called the space-clamped coupling \cite{AC, IJ, WLZ, Yong} represented by the linear terms $C \sum_{j (\neq i)} a_{ij} (x_j (t) - x_i (t))$ in the differential equation of the membrane potential for the $i$-th neuron. The mean field models of Hodgkin-Huxley and FitzHuigh-Nagumo neuron networks possibly with noise were studied in \cite{BF, Dick, QT} replacing the sum of coupling terns shown above by its average. Synchronization and control of the diffusive FitzHugh-Nagumo type or reaction-diffusion type neural networks has also been investigated  \cite{AA, YCY}, which consists of multiple chain-like neurons with the distributed coupling terms $\alpha_i (u_{i-1}(t,x) - u_i (t,x))$ and $\beta_i (v_{i-1}(t,x) - v_i (t,x))$ in the two reaction-diffusion equations for the $i$-th neuron, where $x$ is in the interior of a spatial domain, or by the pointwise pinning-impulse controllers in the interior of a spatial domain. Moreover, in recent years synchronization of chaotic neural networks and stochastic neural networks has also been studied, cf. \cite{Dick, Hetal, SG, WZD},

Synchronization of two coupled Hindmarsh-Rose neurons is shown in \cite{DH, CY}. Very recently we have proved in \cite{PyY} the asymptotic synchronization of the Hindmarsh-Rose neuron networks in a star-like local sense.

In this paper, we shall put forward a new mathematical model of complex neural networks of the partly diffusive Hindmarsh-Rose neurons featuring the dynamic boundary coupling based on the Fick's law and the Kirchhoff's law of the biochemical and bio-electric synapses crossing the boundaries among the neurons in the network. 

We shall formulate the new model first and then analyze the dynamics of the solution semiflows of the partly diffusive Hindmarsh-Rose neural network equations to reach the proof of the main result on the asymptotic exponential synchronization of this type complex neural networks.

\section{\textbf{New Model of Complex Neural Networks with Boundary Coupling}}

\vspace{3pt}
In this work, we propose a new model of complex neural networks in terms of the partly diffusive Hindmarsh-Rose equations,
\beq \bl{cHR}
\begin{split}
	\frac{\pdr u_i}{\pdr t} & = d \gd u_i + au_i^2 - bu_i^3 + v_i - w_i + J,  \quad 1 \leq i \leq N, \\
	\frac{\pdr v_i}{\pdr t} & =  \alpha - v_i - \beta u_i^2,  \quad 1 \leq i \leq N,  \\
	\frac{\pdr w_i}{\pdr t} & = q (u_i - c) - rw_i, \quad 1 \leq i \leq N,
\end{split}
\eeq
for $t > 0,\; x \in \gw \subset \mathbb{R}^{n}$ ($n \leq 3$), an integer $N \geq 2$, and $\gw$ is a bounded domain and its boundary denoted by
$$
	\partial \gw = \G = \bigcup_{j = 1}^N \G_{ij}, \quad \text{for} \;\, i = 1, \cdots, N,
$$
is locally Lipschitz continuous, where $\G_{ij} = \G_{ji}$ and for each given $i \in \{1, \cdots, N\}$ the boundary pieces $\{\G_{ij}: j = 1, \cdots, N\}$  are measurable and mutually non-overlapping. Here $(u_i (t,x), v_i (t,x), w_i (t,x)), \,i = 1, \cdots, N,$ are the state variables for the $i$-th neuron denoted by $\mathcal{N}_i$ in this network. All the neurons are boundary coupled (or some are uncoupled) with the other neurons $\{\mathcal{N}_j: j \neq i \}$ in the network. 

In this Hindmarsh-Rose neuron system \eqref{cHR}, for the $i$-th neuron, the variable $u_i(t,x)$ refers to the membrane electric potential of a neuron cell, the variable $v_i(t, x)$ called the spiking variables represents the transport rate of the ions of sodium and potassium through the fast ion channels, and the variable $w_i(t, x)$ called the bursting variable represents the transport rate across the neuron cell membrane through slow channels of calcium and other ions.

The boundary conditions affiliated with the system \eqref{cHR} are given by
\beq \label{nbc}
	\frac{\pdr u_i}{\pdr \nu} (t, x) + p u_i = pu_j, \quad \text{for}  \; x \in \G_{ij}, \;\,  j \in \{1, \cdots, N\},
\eeq
for $1 \leq i \leq N$, where $\pdr/\pdr \nu$ stands for the normal outward derivative, $p > 0$ is the coupling strength constant. By \eqref{nbc}, $\frac{\pdr u_i}{\pdr \nu} (t, x)  = 0$ for $x \in \G_{ii}, 1 \leq i \leq N$. 

The initial conditions of the equations \eqref{cHR} to be specified will be denoted by
\beq \bl{inc}
	u_i(0, x) = u_i^0 (x), \;\, v_i(0, x) = v_i^0 (x), \;\, w_i (0, x) = w_i^0 (x), \quad x \in \gw, \; 1 \leq i \leq N.
\eeq
All the parameters in this system \eqref{cHR} are positive constants except a reference value of the neuronal membrane potential $c = u_R \in \mR$.

This new mathematical model of the partly diffusive Hindmarsh-Rose neural network is a system of partial-ordinary differential equations featuring the Robin-type boundary condition, which is exactly the scaled combination of the Fick's law and the Kirchhoff's law crossing the coupled cell boundaries and also reflects that synaptic signal stimulations mainly take place in the bio-electric potential $u_i$-equations. 

In this study of the neural network \eqref{cHR}-\eqref{inc}, we define the following Hilbert spaces for each of the subsystems representing the involved single neurons:
$$
	H = L^2 (\gw, \mR^3), \quad  \text{and} \quad E = H^1 (\gw) \times L^2 (\gw, \mR^2)
$$
and the corresponding product Hilbert spaces 
$$
	\mathbb{H} =  [L^2 (\gw, \mR^3)]^{N} \quad \text{and} \quad \mathbb{E} = [H^1 (\gw) \times L^2 (\gw, \mR^2)]^{N}
$$
for the entire system \eqref{cHR}-\eqref{inc}. The norm and inner-product of the Hilbert spaces $\mathbb{H}, \, H$ or $L^2(\gw)$ will be denoted by $\| \, \cdot \, \|$ and $\inpt{\,\cdot , \cdot\,}$, respectively. The norm of $\mathbb{E}$ or $E$ will be denoted by $\| \, \cdot \, \|_E$. We use $| \, \cdot \, |$ to denote the vector norm or the Lebesgue measure of a set in $\mR^n$.

The initial-boundary value problem \eqref{cHR}-\eqref{inc} can be formulated into an initial value problem of the evolutionary equation:
\begin{equation} \label{pb}
	\begin{split}
	\frac{\partial g_i}{\partial t} = A_i g_i &\, + f(g_i), \quad t > 0, \;\, 1 \leq i \leq N, \\[2pt]
	& g_i (0) = g_i^0 \in H.
	\end{split}
\end{equation}
Here $g_i(t) = (u_i(t, \cdot), v_i (t, \cdot ), w_i(t, \cdot))$, whose initial data functions are denoted by $g_i^0 = \text{col}\, (u_i^0, v_i^0, w_i^0),$ for $1 \leq i \leq N$. The nonpositive, self-adjoint operator $A$ for the entire neural network is defined to be $A = diag \,(A_1, \cdots, A_N)$, with the $i$-th block operator being
\begin{equation} \label{opA}
	A_i = 
\begin{bmatrix}
d \gd \quad & 0  \quad & 0 \\[3pt]
0 \quad & - I \quad  & 0 \\[3pt]
0 \quad & 0 \quad & - r I 
\end{bmatrix},
\end{equation}
and its domain is
$$
	D(A) = \{ \text{col}\, (h_1, \cdots, h_m) \in [H^2(\gw) \times L^2 (\gw, \mR^2)]^{N}: \text{\eqref{nbc} satisfied}\}.
$$ 
Due to the continuous Sobolev imbedding $H^{1}(\gw) \hookrightarrow L^6(\gw)$ for space dimension $n \leq 3$ and by the H\"{o}lder inequality, the nonlinear mapping 
\begin{equation} \label{opf}
	f(g_i) =
\begin{pmatrix}
au_i^2 - bu_i^3 + v_i - w_i + J \\[3pt]
\alpha - \beta u_i^2  \\[3pt]
q (u_i - c) \\[3pt]
\end{pmatrix}
: E \longrightarrow H 
\end{equation}
is locally Lipschitz continuous for $1 \leq i \leq N$. 

We shall consider the weak solutions of this initial value problem \eqref{pb}, cf. \cite[Section XV.3]{CV} and the corresponding definition we presented in \cite{PY, CY}. The following underlying proposition can be proved by the Galerkin approximation method.

\begin{proposition} \label{pps}
	For any given initial state $(g_1^0, \cdots , g_N^0) \in \mathbb{H}$, there exists a unique weak solution $(g_1 (t, g_1^0), \cdots, g_N (t, g_N^0))$ local in time $t \in [0, \tau]$, for some $\tau > 0$, of the initial value problem \eqref{pb} formulated from the initial-boundary value problem \eqref{cHR}-\eqref{inc}. The weak solution continuously depends on the initial data and satisfies 
	\begin{equation} \label{soln}
	(g_1, \cdots, g_N)  \in C([0, \tau]; \,\mH) \cap C^1 ((0, \tau); \,\mH) \cap L^2 ([0, \tau]; \,\mE).
	\end{equation}
If the initial state is in $\mE$, then the solution is a strong solution with the regularity
	\begin{equation} \bl{ss}
	(g_1, \cdots, g_N) \in C([0, \tau]; \,\mE) \cap C^1 ((0, \tau); \,\mE) \cap L^2 ([0, \tau]; \,D(A)).
	\end{equation}
\end{proposition}

The basics of infinite dimensional dynamical systems or called semiflow generated by the evolutionary differential equations are referred to \cite{CV, SY, Tm}. A key concept in global dynamics is the absorbing set defned below.

\begin{definition} \label{Dabsb}
	Let $\{S(t)\}_{t \geq 0}$ be a semiflow on a Banach space $\ms{X}$. A bounded set (usually a ball) $B^*$ of $\ms{X}$ is called an absorbing set of this semiflow, if for any given bounded set $B \subset \ms{X}$ there exists a finite time $T_B \geq 0$ depending on $B$, such that $S(t)B \subset B^*$ permanently for all $t \geq T_B$. 
\end{definition}

\section{\textbf{Absorbing Dynamics of the Solution Semiflow}}

We first prove the global existence of weak solutions in time for the problem \eqref{pb} and the existence of an absorbing set of this Hindmarsh-Rose neural network. 

\begin{theorem} \label{Tm}
	For any given initial state $(g_1^0, \cdots , g_N^0) \in \mathbb{H}$, there exists a unique global weak solution $(g_1 (t), \cdots, g_N (t)), \, t \in [0, \infty)$, of the initial value problem \eqref{pb} formulated from the initial-boundary value problem of the boundary coupled Hindmarsh-Rose neural network \eqref{cHR}-\eqref{inc}. 
\end{theorem}

\begin{proof}
	Conduct and sum up the $L^2$ inner-products of the $u_i$-equation with $C_1 u_i(t)$ for $1 \leq i \leq N$, where the constant $C_1 > 0$ is to be chosen. We get		
	\begin{equation*}
	\begin{split}
	&\frac{C_1}{2} \frac{d}{dt} \sum_{i = 1}^N \|u_i\|^2 + C_1 d\, \sum_{i=1}^N \|\nb u_i\|^2  = -  d\,C_1p \, \sum_{i=1}^N \sum_{j=1}^N \int_{\G_{ij}} ( u_i - u_j)^2\, dx \\
	+ &\, \sum_{i=1}^N \int_\gw (C_1 (au_i^3 -bu_i^4  + u_i v_i - u_i w_i +Ju_i) \, dx  \\
	\leq &\, C_1 \int_\gw \left[ \sum_{i=1}^N (au_i^3 -bu_i^4  + u_i v_i - u_i w_i +Ju_i)\right] dx
	\end{split}
	\end{equation*}
by Green's divergence theorem and the coupling boundary condition \eqref{nbc}. Then sum up the $L^2$ inner-products of the $v_i$-equation with $v_i(t)$ for $1 \leq i \leq N$, we obtain
	\begin{equation*}
	\begin{split}
	&\frac{1}{2} \frac{d}{dt} \sum_{i=1}^N \| v_i\|^2 = \int_\gw \left[ \sum_{i=1}^N (\ap v_i - \gb u_i^2 v_i - v_i^2) \right]dx \\
	\leq &\int_\gw \left[ \sum_{i=1}^N (\ap v_i +\frac{1}{2} (\gb^2 u_i^4 + v_i^2) - v_i^2)\right] dx \leq \int_\gw \left[ N\ap^2 +\frac{1}{2} \gb^2 \sum_{i=1}^N u_i^4 - \frac{3}{8} \sum_{i=1}^N v_i^2 \right] dx,
	\end{split}
	\end{equation*}
and sum up the $L^2$ inner-products of the $w_i$-equation with $w_i (t)$ for $1 \leq i \leq N$, similarly we have
	\begin{equation*} 
	\begin{split}
	&\frac{1}{2} \frac{d}{dt} \sum_{i=1}^N \| w_i \|^2 = \int_\gw \left[ \sum_{i=1}^N (q (u_i - c)w_i - rw_i^2) \right] dx  \\
	\leq & \int_\gw \sum_{i=1}^N \left(\frac{q^2}{2r} (u_i - c)^2 + \frac{1}{2} r w_i^2 - r w_i^2\right) dx \leq \int_\gw \left[\frac{q^2}{r} \left(\sum_{i=1}^N u_i^2 + N c^2\right) - \frac{r}{2} \sum_{i=1}^N w_i^2 \right] dx.
	\end{split}
	\end{equation*}
	
Now we choose the positive constant $C_1 = \frac{1}{b} (\gb^2 + 4)$. Then 
	\beq \bl{C1u}
		\int_\gw (- C_1 b u_i^4)\, dx + \int_\gw (\gb^2 u_i^4)\, dx = \int_\gw (-4 u_i^4)\, dx, \quad i= 1, \cdots, N.
	\eeq
Using the Young's inequality, for $1 \leq i \leq N$, we have
	\beq \bl{C3u}
		\int_\gw C_1 au_i^3\, dx \leq \frac{3}{4} \int_\gw u_i^4\, dx + \frac{1}{4}\int_\gw (C_1 a)^4 \, dx \leq \int_\gw u_i^4\, dx + (C_1 a)^4 |\gw|, 
	\eeq
Moreover,
	\beq \bl{uvw}
	\begin{split}
	&C_1 \int_\gw \left( \sum_{i=1}^N (u_i v_i - u_i w_i + Ju_i)\right) dx \\
	\leq &\, \int_\gw \sum_{i=1}^N \left(2(C_1 u_i)^2 + \frac{1}{8} v_i^2 + \frac{(C_1 u_i)^2}{r} + \frac{1}{4} r w_i^2 + C_1 u_i^2 + C_1J^2 \right) dx 	
	\end{split}
	\eeq
	in which we can further deduce
	\beq \bl{ur}
	\begin{split}
	\int_\gw \sum_{i=1}^N \left[2(C_1 u_i)^2 + \frac{(C_1 u_i)^2}{r} + C_1 u_i^2\right] dx \leq \int_\gw \sum_{i=1}^N u_i^4 dx + N |\gw | \left[C_1^2 \left(2 +\frac{1}{r}\right) + C_1\right]^2.
	\end{split}
	\eeq
	Besides we have
	\beq \bl{uq}
	\int_\gw \frac{1}{r} q^2 \left(\sum_{i=1}^N u_i^2\right) dx \leq \int_\gw \left( \sum_{i=1}^N u_i^4\right) dx + \frac{q^4}{r^2}N |\gw|.
	\eeq
Substitute \eqref{C1u} -\eqref{uq} into the first three differential inequalities for sums of $u_i, v_i, w_i$ in this proof and then sum them up to obtain
\beq \label{g2}
	\begin{split}
		&\frac{1}{2} \frac{d}{dt} \left[C_1 \left(\sum_{i = 1}^N \|u_i\|^2\right) +  \left( \sum_{i=1}^N \| v_i\|^2\right) + \left(\sum_{i=1}^N \| w_i \|^2\right) \right] + C_1 d\, \left(\sum_{i=1}^N \|\nb u_i \|^2\right)  \\
		\leq &\, C_1 \int_\gw \left[ \sum_{i=1}^N (au_i^3 -bu_i^4  + u_i v_i - u_i w_i +Ju_i)\right] dx \\
		&+ \int_\gw \left[ N \ap^2 +\frac{1}{2} \gb^2 ( \sum_{i=1}^N u_i^4) - \frac{3}{8} (\sum_{i=1}^mNv_i^2) \right] dx \\[3pt]
		&+ \int_\gw \left[\frac{q^2}{r} \left(\sum_{i=1}^N u_i^2 + N c^2\right) - \frac{r}{2} \left(\sum_{i=1}^N w_i^2\right)\right] dx \\[3pt]
		\leq & \int_\gw (3 - 4)\left(\sum_{i=1}^N u_i^4 \right) dx + \int_\gw \left(\frac{1}{8} - \frac{3}{8}\right) \left(\sum_{i=1}^N v_i^2\right) dx + \int_\gw \left(\frac{1}{4} - \frac{1}{2} \right) r \left( \sum_{i=1}^N w_i^2 \right) dx \\[3pt]
		&+  \, N |\gw | \left( (C_1 a)^4 + C_1 J^2  + \left[C_1^2 \left(2 +\frac{1}{r}\right) + C_1\right]^2 + 2\ap^2 + \frac{q^2 c^2}{r} + \frac{q^4}{r^2} \right) \\[3pt]
		= &\, - \int_\gw \left( \sum_{i=1}^N u_i^4 + \frac{1}{4} \sum_{i=1}^N v_i^2 + \frac{r}{4} \sum_{i=1}^N w_i^2 \right) dx + \frac{1}{2} C_2 N |\gw |, 	\quad \;  t \in I_{max},
	\end{split}
\eeq
where the constant
$$
	C_2 = 2(C_1 a)^4 + 2C_1 J^2  + 2\left[C_1^2 \left(2 +\frac{1}{r}\right) + C_1\right]^2 + 4\ap^2 + \frac{2q^2 c^2}{r} + \frac{2q^4}{r^2}.
$$ 
From \eqref{g2} it follows that
	\beq \label{E1}
	\begin{split}
		&\frac{d}{dt} \left(C_1 \sum_{i = 1}^N \|u_i\|^2 + \sum_{i=1}^N \| v_i\|^2 + \sum_{i=1}^N \| w_i \|^2\right) \\
		+ \,2& \int_\gw \left(\sum_{i=1}^N u_i^4 + \frac{1}{4} \sum_{i=1}^N v_i^2 + \frac{r}{4} \sum_{i=1}^N w_i^2 \right) dx \leq C_2 N |\gw|, 
	\end{split}
	\eeq
for $t \in I_{max} = [0, T_{max})$, which is the maximal time interval of solution existence. Note that in the first part of the integral term of \eqref{E1} we have $ \frac{1}{4} \left(C_1 u_i^2 - \frac{C_1^2}{16}\right) \leq u_i^4, 1 \leq i \leq N$. Then \eqref{E1} yields the following differential inequality
	\beq \bl{E2}
	\begin{split}
	 &\frac{d}{dt} \left[C_1 \sum_{i = 1}^N \|u_i\|^2 +  \sum_{i=1}^N \| v_i\|^2 + \sum_{i=1}^N \| w_i \|^2 \right] \\
		&+ \, r^* \left[C_1  \sum_{i = 1}^N \|u_i\|^2 + \sum_{i=1}^N \| v_i\|^2 + \sum_{i=1}^N \| w_i \|^2 \right]  \\
	&\leq \frac{d}{dt} \left[C_1 \sum_{i = 1}^N \|u_i\|^2 +  \sum_{i=1}^N \| v_i\|^2 + \sum_{i=1}^N \| w_i \|^2 \right] \\
		&+ \, \frac{1}{2} \int_\gw \left( \sum_{i=1}^N u_i^2 +  \sum_{i=1}^N v_i^2 + r \sum_{i=1}^N w_i^2 \right) dx \leq \left(C_2 + \frac{C_1^2}{32}\right) N|\gw |,
	\end{split}
	\eeq
where $r^* = \frac{1}{2} \min \{1, r\}$. Apply the Gronwall inequality to \eqref{E2}. Then we obtain the following bounding estimate of the weak solutions,
	\beq \label{dse}
	\begin{split}
		\sum_{i=1}^N \|g_i (t, g_i^0)\|^2 &= \sum_{i=1}^N \left[ \|u_i (t)\|^2 + \| v_i (t)\|^2 + \| w_i (t) \|^2\right] \\
		&\leq \frac{\max \{C_1, 1\}}{\min \{C_1, 1\}}e^{- r^* t} \sum_{i=1}^N \|g_i^0\|^2 + \frac{M}{\min \{C_1, 1\}} |\gw | \\
		&\leq \frac{\max \{C_1, 1\}}{\min \{C_1, 1\}} \sum_{i=1}^N \|g_i^0\|^2 + \frac{M}{\min \{C_1, 1\}} |\gw |
	\end{split}
	\eeq 
for $t \in I_{max} = [0, T_{max}) = [0, \infty)$, where 
\beq \bl{M}
	M = \frac{N}{r^*}\left(C_2 + \frac{C_1^2}{32}\right).
\eeq
The uniform bound estimate \eqref{dse} on the time interval of solution existence shows that the weak solution $g_i(t, x), 1 \leq i \leq N,$ will never blow up at any finite time. Therefore, for any initial data in $\mH$, the weak solution of the initial value problem \eqref{pb} of this neural network \eqref{cHR}-\eqref{inc} exists in $\mH$ for $t \in [0, \infty)$.
\end{proof}

The global existence and uniqueness of the weak solutions and their continuous dependence on the initial data enable us to define the solution semiflow $\{S(t): \mH \to \mH\}_{t \geq 0}$ of this boundary coupled Hindmarsh-Rose neuron network system \eqref{cHR}-\eqref{inc} on the space $\mH$ to be
\beq \bl{HRS}
	S(t): (g_1^0, \cdots, g_N^0) \longmapsto (g_1(t, g_1^0), \cdots, g_N(t, g_N^0)), \quad  t \geq 0.
\eeq
We call this semiflow $\{S(t)\}_{t \geq 0}$ the \emph{boundary coupling Hindmarsh-Rose semiflow}. The next result is the absorbing property of this semiflow which is leveraged to show the asymptotic synchronization of this complex neural network in the next section.

\begin{theorem} \label{Hab}
	There exists an absorbing set for the boundary coupling Hindmarsh-Rose semiflow $\{S(t)\}_{t \geq 0}$ in the space $\mH$, which is the bounded ball 
\beq \label{abs}
	B^* = \{ h \in \mH: \| h \|^2 \leq Q\}
\eeq 
where the constant
\beq \bl{Q}
	Q =  \frac{2M |\gw |}{\min \{C_1, 1\}} =  \frac{2N }{r^*\min \{C_1, 1\}}\left(C_2 + \frac{C_1^2}{32}\right) |\gw |.
\eeq
\end{theorem}

\begin{proof}
This is the consequence of the uniform estimate \eqref{dse} in Theorem \ref{Tm} because
	\beq \label{lsp}
	\limsup_{t \to \infty} \, \sum_{i=1}^N \|g_i(t, g_i^0)\|^2 < Q = \frac{2M |\gw |}{\min \{C_1, 1\}} 
	\eeq
	for all weak solutions of \eqref{pb} with any initial data $(g_1^0, \cdots, g_N^0)$ in $\mH$. Moreover, for any given bounded set $B = \{h \in \mH: \|h \|^2 \leq \rho \}$ in $\mH$, there exists a finite time 
	\beq \label{T0B}
	T_0 (B) = \frac{1}{r^*} \log^+ \left(\rho \, \frac{\max \{C_1, 1\}}{M |\gw|}\right)
	\eeq
	such that all the solution trajectories started from the set $B$ will permanently enter the bounded ball $B^*$ shown in \eqref{abs} for $t \geq T_0(B)$.  According to Definition \ref{Dabsb}, the theorem is proved.
\end{proof}

\section{\textbf{Synchronization of the Hindmarsh-Rose Neuron Neuwork}} 

We make a definition to describe synchronization dynamics for mathematical models of biological neural networks.
\begin{definition} \bl{DaD}
	For the dynamical system generated by a model differential equation such as \eqref{pb} for a neural network with whatever type of coupling, define the \emph{asynchronous degree} in a state space $\ms{X}$ to be
	$$
	deg_s (\ms{X})= \sum_{j} \sum_{k} \, \sup_{g_j^0, \, g_k^0 \in \ms{X}} \, \left\{\limsup_{t \to \infty} \, \|g_j (t) - g_k(t)\|_{\ms{X}}\right\},
	$$ 
where $g_j(t)$ and $g_k(t)$ are any two solutions of the model differential equation for two neurons with the initial states $g_j^0$ and $g_k^0$, respectively. Then the neural network is said to be asymptotically synchronized in the space $\ms{X}$, if $deg_s (\ms{X}) = 0$.
\end{definition}

We shall prove the main result on the asymptotic synchronization of the boundary coupled Hindmarsh-Rose neural network described by \eqref{cHR}-\eqref{inc} in the space $H$. This result provides a quantitative threshold for the boundary coupled Hindmarsh-Rose neural network to reach synchronization.

Denote by $U_{ij} (t) = u_i(t) - u_j (t), V_{ij} (t) = v_i(t) - v_j(t), W_{ij} (t) = w_i(t) - w_i(t)$, for $i, j = 1, \cdots, N$. For any given initial states $g_i^0, \cdots, g_N^0$ in the space $H$, the difference between the solutions of the modrel equation \eqref{pb} associated with the coupled (or uncoupled) neurons $\mathcal{N}_i$ and $\mathcal{N}_j$ is
	$$
	g_i (t, g_i^0) - g_j (t, g_j^0) = \text{col}\, (U_{ij}(t), V_{ij}(t), W_{ij}(t)), \quad t \geq 0.
	$$
	
	By subtraction of the corresponding three equations of the $j$-th neuron from the corresponding equations of the $i$-th neuron in \eqref{cHR}, we obtain the \emph{differencing} Hindmarsh-Rose equations as follows. For $i, j = 1, \cdots, N$,
\beq \bl{dHR}
	\begin{split}
		\frac{\pdr U_{ij}}{\pdr t} & = d \gd U_{ij} +  a(u_i + u_j)U_{ij} - b(u_i^2 + u_i u_j + u_j^2)U_{ij} + V_{ij} - W_{ij},  \\
		\frac{\pdr V_{ij}}{\pdr t} & =  - V_{ij} - \beta (u_i +  u_j)U_{ij},    \\
		\frac{\pdr W_{ij}}{\pdr t} & = q U_{ij} - r W_{ij}.
	\end{split}
\eeq

Here is the main result on the synchronization of the complex Hindmarsh-Rose neural network with boundary coupling.
\begin{theorem} \bl{ThM}
	If the following threshold condition for stimulation signal strength of the boundary coupled Hindmarsh-Rose neural network is satisfied,  
\beq \bl{SC}
	p\, \liminf_{t \to \infty}\,\sum_{1 \leq i < j \leq N} \int_{\G_{ij}} U_{ij}^2(t, x)\, dx > R \, |\gw|,   
\eeq
for any given initial conditions $(g_1^0, \cdots, g_N^0) \in \mathbb{H}$, where the constant $R > 0$ is 
\beq \bl{R}
	R = \frac{N^2 (N-1)}{r^* \min \{C_1, 1\}}\left[\frac{C_1^2}{16} + 2C_2\right] \left[\eta_2\, d \,|\gw | + \left[\frac{8\beta^2}{b} + \frac{2a^2}{b} + \frac{b}{16\beta^2 r} \left[q - \frac{8\beta^2}{b}\right]^2\right] \right]
\eeq
with $C_1 = \frac{1}{b} (\beta^2 + 4)$, $\eta_2 > 0$ being the constant in the generalized Poincar\'{e} inequality \eqref{Pcr}, and
\beq  \bl{C2}
	C_2 = 2(C_1 a)^4 + 2C_1 J^2  + 2\left[C_1^2 \left(2 +\frac{1}{r}\right) + C_1\right]^2 + 4\ap^2 + \frac{2q^2 c^2}{r} + \frac{2q^4}{r^2},
\eeq
then the boundary coupled Hindmarsh-Rose neural network modeled by \eqref{pb} is asymptotically synchronized in the space $H$ at a uniform exponential rate.
\end{theorem}

\begin{proof}
	We go through three steps to estimate the solutions of the differencing equations \eqref{dHR} and prove this result.
	
	Step 1. Take the $L^2$ inner-products of the first equation in \eqref{dHR} with $GU_{ij} (t)$, the second equation with $V_{ij} (t)$, and the third equation with $W_{ij} (t)$, where $G > 0$ is to be chosen. Then sum them up and use Young's inequalities to get
\beq \bl{eG}
	\begin{split}
		&\frac{1}{2} \frac{d}{dt} (G\|U_{ij} (t)\|^2 + \|V_{ij} (t)\|^2 + \|W_{ij} (t)\|^2) + d G \|\nb U_{ij} (t)\|^2  + \|V_{ij} (t)\|^2 + r\, \|W_{ij} (t)\|^2 \\[6pt]
		= &\,\int_{\G} GU_{ij} \frac{\pdr U_{ij}}{\pdr \nu} \, dx +  \int_\gw G \left[a(u_i + u_j)U_{ij}^2 - b(u_i^2 + u_i u_j + u_j^2) U_{ij}^2\right] dx \\[5pt]
		+ &\, \int_\gw \left[ G U_{ij} V_{ij} -\beta (u_i +  u_j)U_{ij} V_{ij}+ (q - G)U_{ij} W_{ij} \right] dx \\[5pt]
		\leq &\,\int_{\G} GU_{ij} \frac{\pdr U_{ij}}{\pdr \nu}\, dx + \left[ G^2 + \frac{1}{2r} (q - G)^2\right] \|U_{ij} (t)\|^2 + \frac{1}{4}\|V_{ij} (t)\|^2 + \frac{r}{2}\|W_{ij} (t)\|^2 \\[5pt]
		+ &\int_\gw \left[G a(u_i + u_j)U_{ij}^2  -\beta (u_i +  u_j)U_{ij} V_{ij} - G b\,(u_i^2 + u_i u_j + u_j^2) U_{ij}^2 \right] dx, \quad t > 0.
		\end{split}
\eeq

By the the boundary coupling condition \eqref{nbc}, the boundary integral in \eqref{eG} turns out to be
\beq \bl{bdt}
	\begin{split}
	\int_{\G} GU_{ij} \frac{\pdr U_{ij}}{\pdr \nu}\, dx &= -\, Gp \left[\sum_{k = 1}^N  \int_{\G_{ik}}\, (u_i - u_k) U_{ij} \, dx - \sum_{k=1}^N \int_{\G_{jk} }\, (u_j - u_k)] U_{ij} \, dx\right] \\
	&= -Gp K_{ij},
	\end{split}
\eeq	
where we set the quantity
\beq \bl{K}
	K_{ij} = \sum_{k = 1}^N \int_{\G_{ik}} (u_i - u_k) (u_i - u_j) \, dx - \sum_{k = 1}^N \int_{\G_{jk}}  (u_j - u_k) (u_i - u_j)] \,dx.
\eeq

Then we estimate another integral term on the right-hand side of \eqref{eG}, 
\beq \bl{nlt}
	\begin{split}
		&\int_\gw \left(G a(u_i + u_j)U_{ij}^2  -\beta (u_i +  u_I)U_{ij} V_{ij} - G b\,(u_i^2 + u_i u_j + u_j^2) U_{ij}^2 \right) dx \\
		\leq &\, \int_\gw \left(G a(u_i + u_j)U_{ij}^2  - \beta (u_i +  u_j)U_{ij}V_{ij} - \frac{G b}{2}(u_i^2 + u_j^2) U_{ij}^2 \right) dx \\
		\leq &\, \int_\gw \left(G a(u_i + u_j)U_{ij}^2  + 2\beta^2 (u_i^2 +  u_j^2)U_{ij}^2 + \frac{1}{4} V_{ij}^2 - \frac{G b}{2}(u_i^2 + u_j^2) U_{ij}^2 \right) dx.
	\end{split}
\eeq
Now we choose the constant multiplier $G$ to be
	\beq \bl{lbd}
	G = \frac{8 \beta^2}{b} > 0.
	\eeq
By this choice, \eqref{nlt} is reduced to
	\beq \bl{me}
	\begin{split}
		&\int_\gw \left(G a\,(u_i + u_j)U_{ij}^2  -\beta (u_i +  u_j)U_{ij} V_{ij} - G b\,(u_i^2 + u_i u_j + u_j^2) U_{ij}^2 \right) dx \\
		\leq &\, \int_\gw \left(G a(u_i + u_j)U_{ij}^2 + \frac{1}{4} V_{ij}^2 - \frac{G b}{4}(u_i^2 + u_j^2) U_{ij}^2 \right) dx \\
		= &\, \frac{1}{4} \|V_{ij} (t)\|^2 +  \int_\gw \left( a(u_i + u_j) - \frac{b}{4}(u_i^2 + u_j^2) \right) GU_{ij}^2\, dx \\
		= &\, \frac{1}{4} \|V_{ij} (t)\|^2 +  \int_\gw \left[\frac{2a^2}{b} - \left(\frac{a}{b^{1/2}} - \frac{b^{1/2}}{2}\, u_i \right)^2 - \left(\frac{a}{b^{1/2}} - \frac{b^{1/2}}{2}\, u_j \right)^2 \right] GU_{ij}^2\, dx \\
		\leq &\, \frac{1}{4} \|V_{ij} (t)\|^2 + \frac{2G a^2}{b} \|U_{ij} (t)\|^2.
	\end{split}
	\eeq

Substitute the estimates \eqref{bdt} for the boundary integral and \eqref{me} for the domain integral into \eqref{eG}. We obtain
\begin{equation} \bl{meq}
	\begin{split}
	&\frac{1}{2} \frac{d}{dt} (G \|U_{ij} (t)\|^2 + \|V_{ij} (t)\|^2 + \|W_{ij} (t)\|^2) + GP K_{ij} \\
	&+ dG  \|\nb U_{ij} (t)\|^2 + \frac{1}{2}\|V_{ij} (t)\|^2 + \frac{r}{2}\, \|W_{ij} (t)\|^2 \\
	\leq &\, \left(G^2 + \frac{2G a^2}{b} + \frac{1}{2r} (q - G)^2\right) \|U_{ij} (t)\|^2, \quad t > 0, \quad 1 \leq i, j \leq N.
	\end{split}
\end{equation}

Step 2. We now utilize the generalized Poincar\'{e} inequality \cite{S} to treat the gradient term on the left-hand side of \eqref{meq}.  According to the generalized Poincar\'{e} inequality, there exist positive constants $\eta_1$ and $\eta_2$ depending only on the spatial domain $\gw$ and its dimension such that 
\beq \bl{Pcr}
	\eta_1 \|U_{ij} (t)\|^2 \leq \|\nb U_{ij} (t)\|^2 + \eta_2 \left(\int_\gw U_{ij} (t, x)\, dx\right)^2, \quad 1 \leq i, j \leq N.
\eeq
On the other hand, \eqref{Q} and \eqref{lsp} confirm that 
\beq \bl{Lsup}
	\limsup_{t \to \infty}\, \sum_{i=1}^N \|g_i(t, g_i^0)\|^2 < Q = \frac{2N}{r^* \min \{C_1, 1\}}\left(C_2 + \frac{C_1^2}{32}\right) |\gw |.
\eeq	
Also note that for all $1 \leq i, j \leq N$,
$$
	 \|U_{ij} (t)\|^2 \leq  2 (\|u_i (t)\|^2 + \|u_j (t)\|^2) \leq 2 \sum_{i=1}^N \|g_i(t, g_i^0)\|^2.
$$
Thus for any given bounded set $B \subset H$ and any initial data $g_i^0, g_j^0 \in B, 1 \leq i, j \leq N$, there is a finite time $T_B > 0$ depending on $B$ only such that
\beq \bl{gbd}
	\|U_{ij} (t)\|^2 \leq 2Q = \frac{4N}{r^* \min \{C_1, 1\}}\left(C_2 + \frac{C_1^2}{32}\right) |\gw |, \quad \text{for} \; t > T_B.
\eeq
By \eqref{meq}, \eqref{Pcr} and \eqref{gbd}, for any given bounded set $B \subset H$ and any initial data $g_i^0, g_j^0 \in B, 1 \leq i, j \leq N$, there exists a finite time $T_B > 0$ such that for all $t > T_B$ we have
\beq \bl{Keq}
	\begin{split}
	& \frac{d}{dt} (G \|U_{ij} (t)\|^2 + \|V_{ij} (t)\|^2 + \|W_{ij} (t)\|^2) + 2Gp \,K_{ij} \\[6pt]
	& + 2\eta_1 dG\, \|U_{ij} (t)\|^2 + \|V_{ij} (t)\|^2 + r \|W_{ij} (t)\|^2 \\[2pt]
	\leq &\, 2\eta_2\, dG \left(\int_\gw U_{ij} (t, x)\, dx\right)^2 + 2 \left(G^2 + \frac{2G a^2}{b} + \frac{1}{2r} (q - G)^2\right) \|U_{ij} (t)\|^2  \\[3pt]
	\leq &\, 2\eta_2\, dG \|U_{ij} (t)\|^2 |\gw | + 2\left(G^2 + \frac{2G a^2}{b} + \frac{1}{2r} (q - G)^2\right) \|U_{ij} (t)\|^2 \\
	= &\, 2\|U_{ij} (t)\|^2 \left[ \eta_2\, dG |\gw | + \left(G^2 + \frac{2G a^2}{b} + \frac{1}{2r} (q - G)^2\right)  \right] \\
	\leq &\, \frac{4N}{r^* \min \{C_1, 1\}}\left(2C_2 + \frac{C_1^2}{16}\right) |\gw | \left[\eta_2\, dG |\gw | + \left[G^2 + \frac{2G a^2}{b} + \frac{1}{2r} (q - G)^2\right] \right].
	\end{split}
\eeq
The differential inequality \eqref{Keq} is exactly written as 
\beq  \bl{Synq}
	\begin{split}
	& \frac{d}{dt} (G \|U_{ij} (t)\|^2 + \|V_{ij} (t)\|^2 + \|W_{ij} (t)\|^2) + 2Gp \,K_{ij} \\[3pt]
	+ 2 \eta_1 dG &\, \|U_{ij} (t)\|^2 + \|V_{ij} (t)\|^2 + r \|W_{ij} (t)\|^2 \leq 4GR\,|\gw |,  \quad  t > T_B.
	\end{split}
\eeq
In \eqref{Synq}, $G = 8\beta^2/b$ in \eqref{lbd} and the constant $R$ given in \eqref{R} are independent of initial data. 

Step 3. Now we need to treat the key term $2Gp K_{ij}$ in the differential inequality \eqref{Synq}. Let $\psi_{ik}$ be the characteristic function on the boundary piece $\G_{ik}, 1 \leq i,k \leq N$,
$$
	\psi_{ik} (x) = \begin{cases}
		1, \quad & \text{if $x \in \G_{ik}$} , \\
		0, \quad & \text{if $x \in \G \backslash \G_{ik}$.}
		\end{cases} 
$$	
From \eqref{K} we see that 
\begin{equation} \bl{EK}
	\begin{split}
	&\sum_{i,j} K_{ij} = \sum_{i, j} \sum_{k = 1}^N \int_{\G_{ik}} (u_i - u_k) (u_i - u_j) \, dx - \sum_{i,j} \sum_{k = 1}^N \int_{\G_{jk}}  (u_j - u_k) (u_i - u_j) \,dx \\
	= &\,  \sum_{i,j} \sum_{k = 1}^N\int_{\G} \psi_{ik} (u_i - u_k) (u_i - u_j)\, dx -  \sum_{i,j} \sum_{k = 1}^N \int_{\G} \psi_{jk} (u_j - u_k) (u_i - u_j)\, dx \\
	= &\, \sum_{i, j} \int_{\G} \left(\sum_{k=1}^N \psi_{ik}(u_i -u_k)\right) (u_i - u_j) dx  - \sum_{i, j} \int_{\G} \left(\sum_{k=1}^N \psi_{jk} (u_j -u_k)\right) (u_i - u_j) dx \\
	= &\, \sum_{i, j} \int_{\G} \left(u_i - \sum_{k=1}^N \psi_{ik} u_k \right) (u_i - u_j)\, dx - \sum_{i, j} \int_{\G} \left(u_j - \sum_{k=1}^N \psi_{jk} u_k \right) (u_i - u_j)\, dx \\
	= &\, \sum_{i, j} \int_{\G} \left(u_i -  \widetilde{u}_i \right) (u_i - u_j)\, dx - \sum_{i, j} \int_{\G} \left(u_j -  \widetilde{u}_j \right) (u_i - u_j)\, dx \\
	= &\, \sum_{i, j} \int_{\G} \left(u_i -  u_j \right) (u_i - u_j)\, dx - \sum_{i, j} \int_{\G} \left(\widetilde{u_i} -  \widetilde{u}_j \right) (u_i - u_j)\, dx \\
	= &\, \sum_{i=1}^N \sum_{j=1}^N \int_{\G} (u_i - u_j)^2\, dx - \sum_{i=1}^N \sum_{j=1}^N  \int_{\G} \left(\widetilde{u_i} -  \widetilde{u}_j \right) (u_i - u_j)\, dx \\
	= &\, \sum_{i=1}^N \sum_{j=1}^N \int_{\G} (u_i - u_j)^2\, dx, \quad \text{for} \;\; t > 0,
	\end{split}
\end{equation}
where we denote by  $\widetilde{u}_i, 1 \leq i \leq N$, the distributed all $u_k$-components of the solutions $g_k(t, g_k^0)$ on the $i$-th neuron's coupling decomposition of the boundary $\G = \bigcup \G_{ik}$, i.e.
$$
	 \widetilde{u}_i = \sum_{k=1}^N \psi_{ik} \, u_k = u_1\mid_{\G_{i1}} + \cdots + u_N \mid_{\G_{iN}}, 
$$
and we have
$$
	\sum_{i=1}^N \sum_{j=1}^N  \int_{\G} \left(\widetilde{u_i} -  \widetilde{u}_j \right) (u_i - u_j)\, dx = \sum_{i=1}^N \left( \sum_{j < i} + \sum_{j > i} \right) \int_{\G} \left(\widetilde{u_i} -  \widetilde{u}_j \right) (u_i - u_j)\, dx = 0.
$$
Note that $K_{ii} = 0$ and $K_{ij} = K_{ji}$. Substitute \eqref{EK} into \eqref{Synq} and then sum up these differential inequalities for all $1 \leq i < j \leq N$. Then we have

\beq  \bl{SQ}
	\begin{split}
	& \frac{d}{dt} \sum_{1 \leq i < j \leq N} (G \|U_{ij} (t)\|^2 + \|V_{ij} (t)\|^2 + \|W_{ij} (t)\|^2) + 2Gp \sum_{1 \leq i < j \leq N} \int_\G (u_i - u_j)^2\, dx \\
	+ &\, \sum_{1 \leq i < j \leq N} \left(2 \eta_1 dG \|U_{ij} (t)\|^2 + \|V_{ij} (t)\|^2 + r \|W_{ij} (t)\|^2\right) \leq 2N(N-1)GR\,|\gw |,  \quad  t > T_B.
	\end{split}
\eeq

Under the threshold condition \eqref{SC} of this theorem, the stimulation signal strength of this boundary coupled neural network satisfies the threshold crossing inequality: for any initial state $(g_1^0, \cdots , g_N^0) \in \mathbb{H}$, there is a finite time $\tau = \tau (g_1^0, \cdots , g_N^0) > 0$  such that 
	\beq \bl{pR}
	2Gp\, \sum_{1 \leq i < j \leq N} \int_{\G_{ij}} U_{ij}^2(t, x)\, dx > 2GR\,|\gw |,  \quad \text{for} \;\; t > \tau (g_1^0, \cdots, g_N^0).
	\eeq
It follows from \eqref{SQ} and \eqref{pR} that
\beq \bl{QQ}
	\|U_{ij} (t)\|^2 + \|V_{ij} (t)\|^2 + \|W_{ij} (t)\|^2 \leq \|g_i (t, g_i^0)\|^2 + \|g_j (t, g_j^0)\|^2 < 2Q, \quad \text{for} \;\; t > \tau,
\eeq
and
\beq \bl{Swq}
	\begin{split}
	& \frac{d}{dt} \sum_{1 \leq i < j \leq N} \left(G\|U_{ij} (t)\|^2 + \|V_{ij} (t)\|^2 + \|W_{ij} (t)\|^2\right)  \\
	+ \min &\, \{2\eta_1 d, 1, r\} \sum_{1 \leq i < j \leq N} \left(G\|U_{ij} (t)\|^2 + \|V_{ij} (t)\|^2 + \|W_{ij} (t)\|^2\right) < 0, \quad t > \tau,
	\end{split}
\eeq
where $\tau = \tau (g_1^0, \cdots, g_N^0)$ is shown in \eqref{pR}. Apply the Gronwall inequality to \eqref{Swq} combined with \eqref{QQ} to obtain
\beq \bl{FR}
	\begin{split}
		&\sum_{1 \leq i < j \leq N} \min \{G, 1\} \|g_i (t, g_i^0) - g_j (t, g_j^0)\|^2_H \\
		= &\, \sum_{1 \leq i < j \leq N} \min \{G, 1\} \left( \|U_{ij} (t)\|^2 + \|V_{ij} (t)\|^2 + \|W_{ij} (t)\|^2\right)   \\
		\leq &\, \sum_{1 \leq i < j \leq N} \left(G \|U_{ij} (t)\|^2 + \|V_{ij} (t)\|^2 + \|W_{ij} (t)\|^2\right)   \\
		\leq &\, e^{- \mu (t - \tau)}  \sum_{1 \leq i < j \leq N} \left(G \|U_{ij} (\tau)\|^2 + \|V_{ij} (\tau)\|^2 + \|W_{ij} (\tau)\|^2\right) \\
		\leq &\, 2e^{- \mu (t - \tau)} \max \{G, 1\} \,Q \to 0, \quad \text{as} \;\; t \to \infty,
	\end{split}
\eeq
where $\mu = \min \{2\eta_1 d, 1, r\}$ is the uniform exponential convergence rate. It shows that
\beq \bl{gij}
	deg_s (\text{H})= \sum_{i = 1}^N \sum_{j=1}^N \,\sup_{g_i^0, g_j^0 \in H} \left\{ \limsup_{t \to \infty} \|g_i (t, g_i^0) - g_j (t, g_j^0)\|_H \right\} = 0.
\eeq
According to Definitoon \ref{DaD}, this boundary coupled Hindmarsh-Rose neural network generated by \eqref{pb} is asymptotically synchronized in the space $H = L^2 (\gw, \mR^3)$ at a uniform exponential rate. The proof is completed.
\end{proof}

The main theorem in this paper provides a sufficient condition for realization of the asymptotic synchronization of this boundary coupled complex neural network. The biological interpretation of the threshold condition for synchronization is that the product of the \emph{boundary coupling strength} represented by the coupling coefficient $p$ and the \emph{dynamic stimulation signals} represented by $\liminf_{t \to \infty}\, \sum_{i, j} \int_{\G_{ij}} U_{ij}^2(t, x)\, dx$ for all the coupled neurons in the network exceeds the threshold constant $R|\Omega|$ explicitly expressed by the biological and mathematical parameters. 

As a corollary of Theorem \ref{ThM}, the same proof of \eqref{Keq} through \eqref{FR} shows that the complex neural networks presented can also be partly synchronized if the threshold condition \eqref{SC} is satisfied only for a subset of the neurons in the network.

\bibliographystyle{amsplain}

\end{document}